\documentclass[a4paper,leqno,12pt]{amsart}
\usepackage{amsfonts,amssymb,verbatim,amsmath,amsthm,latexsym,textcomp,amscd}
\usepackage{latexsym,amsfonts,amssymb,epsfig,verbatim}
\usepackage{amsmath,amsthm,amssymb,latexsym,graphics,textcomp}
\usepackage{graphicx}
\usepackage{color}
\usepackage{tikz}
\usepackage{url}
\usepackage{enumerate}
\usepackage[mathscr]{euscript}
\usepackage{txfonts}
\usepackage{euscript}
\usepackage[all]{xy}
\usepackage{epic,eepic}
\input xy
\xyoption{all}

\usepackage{hyperref}

\theoremstyle{plain}
\newtheorem{theorem}{Theorem}[section]

\newtheorem{lemma}[theorem]{Lemma}
\newtheorem{cor}[theorem]{Corollary}
\newtheorem{prop}[theorem]{Proposition}

\newtheorem{example}[theorem]{Example}
\newtheorem{remark}[theorem]{Remark}

\theoremstyle{definition}
\newtheorem{defn}[theorem]{Definition}
\newtheorem{rmk}[theorem]{Remark}
\newtheorem{exam}[theorem]{Example}

\begin{document}
\title{Cup-product in Hom-Leibniz cohomology and Hom-Zinbiel algebras}
\author{Ripan Saha}
\email{ripanjumaths@gmail.com}
\address{Department of Mathematics,
Raiganj University, Raiganj, 733134,
West Bengal, India.}

\begin{abstract}
We define a cup-product in Hom-Leibniz cohomology and show that the cup-product satisfies the graded Hom-Zinbiel relation.
\end{abstract}
\maketitle
\section{Introduction}
The concept of Hom-Lie algebras was introduced by Hartwig, Larsson and Silvestrov in the study of algebraic structures describing some q-deformations of the Witt and the Virasoro algebras \cite{HLS}. Hom-Leibniz algebras was introduced by Makhlouf and Silvestrov generalizing the notion of Leibniz algebras \cite{MS08}. Hom-Leibniz algebra is obtained by twisting the Leibniz identity by a self linear map $\alpha$. A Hom-Leibniz algebra is same as a Leibniz algebra when $\alpha$ is the identity map.

There is a cup-product on the graded Hochschild cohomology $H^\ast (A,~A)$ of an associative algebra $A$ and with this cup-product $H^\ast (A,~A)$ carries a Gerstenhaber algebra structure \cite{G1}. A cup-product and Gerstenhaber algebra on Hom-associative algebras have been studied in \cite{AD}. Loday defined a cup-product on the graded cohomology groups of the Leibniz algebras \cite{L3}. The cup-product on the cohomology satisfies the graded Zinbiel algebra relation. A cup-product on the equivariant cohomology of Leibniz algebra equipped with an action of a finite group has been studied in \cite{MS}.

Exploring the similarities and dissimilarities between the category of Leibniz algebras and the larger category of Hom-Leibniz algebras is an interesting theme in the literature. This paper contributes to the development of the non-associative algebra theory in this direction. The main purpose of this paper is to extend Loday\rq s result to Hom-Leibniz algebra case. We introduce cohomology for Hom-Leibniz algebra $(L,[.,.],\alpha)$ with coefficients in a Hom-associative, commutative algebra $(A,\mu,\alpha_1)$ and define a cup-product on the graded cohomology $HL^\ast _{\alpha, \alpha_1}(L;~A)$ of the Hom-Leibniz algebra $L$
\begin{align*}
\cup : HL^n_{\alpha, \alpha_1}(L;~A)\times HL^m_{\alpha, \alpha_1}(L;~A)\rightarrow HL^{n+m}_{\alpha, \alpha_1}(L;~A),
\end{align*}
which satisfies the following graded algebra relation:
\begin{align*}
\alpha_1(a\cup b)\cup c= a\cup\alpha_1(b\cup c)+(-1)^{|b||c|}a\cup\alpha_1(c\cup b),
\end{align*}
where $a\in HL^n_{\alpha, \alpha_1}(L;~A),~b\in HL^m_{\alpha, \alpha_1}(L;~A),~c\in HL^r_{\alpha, \alpha_1}(L;~A)$.

\section{Preliminaries}
In this section, we recall some definitions and basic facts of Hom-algebras and shuffle algebras. In this paper, $k$ always denotes any field.
\begin{defn}\label{hom-ass-defn}
	A Hom-associative algebra over $k$ is a triple $(A, \mu, \alpha_1)$ consists of a $k$-vector space $A$ together with a $k$-bilinear map
	$\mu : A \times A \rightarrow A$ and a $k$-linear map $\alpha_1 : A \rightarrow A$ satisfying $\alpha_1 (\mu (x,y)) = \mu(\alpha_1 (x), \alpha_1 (y))$ and
	\begin{align*}\label{hom-ass-cond}
	\mu ( \alpha_1 (x) , \mu ( y , z) ) = \mu ( \mu (x , y) , \alpha_1 (z)), ~~ \text{ for all } x, y, z \in A.
	\end{align*}
\end{defn}
\begin{exam} \label{hom ass exam}
Let $A$ be a two dimensional vector space with basis $\{a_1, a_2\}$. Define a multiplication $\mu : A \times A \rightarrow A$ by
$$ \mu (a_i, a_j) = \begin{cases} a_1 & \mbox{ if } (i,j) = (1,1)\\
a_2 & \mbox{ if } (i,j) \neq (1,1).
\end{cases} $$
A linear map $\alpha_1 :A \rightarrow A$ is defined by $\alpha_1 (a_1) = a_1 - a_2$ and $\alpha_1 (a_2) = 0$. Then the triple $(A, \mu, \alpha_1)$ defines a Hom-associative algebra structure.
\end{exam}
\begin{defn}
A Hom-Lie algebra is a triple $(A, \mu, \alpha)$
consisting of a linear space $A$, a skew-bilinear map $\mu: V\times V \rightarrow V$ and
 a linear space homomorphism $\alpha: A \rightarrow A$
 satisfying the Hom-Jacobi identity
$$
 \circlearrowleft_{x,y,z}{\mu(\mu(x,y),\alpha(z))}=0,
$$
for all $x, y, z$ in $A$, where
$\circlearrowleft_{x,y,z}$ denotes summation over
the cyclic permutations on $x,y,z$.
\end{defn}

A Hom-Lie algebra whose endomorphism $\alpha$ is the identity is a Lie algebra.
\begin{defn}
A Hom-Leibniz algebra is a $k$-linear vector space $L$ equipped with a $k$-bilinear map $[.,.]:L\times L\to L$ and a $k$-linear map $\alpha:L\to L$ satisfying the identity:
\begin{center}
$[\alpha(x),[y,z]]=[[x,y],\alpha(z)]-[[x,z],\alpha(y)].$
\end{center}
\end{defn}
We denote a Hom-Leibniz algebra as $(L,[.,.],\alpha)$. A Hom-Leibniz algebra $(L,[.,.],\alpha)$ is called multiplicative if it satisfies $[\alpha(x),\alpha(y)]=\alpha([x,y])$.

A homomorphism between Hom-Leibniz algebras $(L_1,[.,.]_1,\alpha_1)$ and $(L_2,[.,.]_2,\alpha_2)$ is a $k$-linear map $\phi: L_1\to L_2$ which satisfies $\phi([x,y]_1)=[\phi(x),\phi(y)]_2$ and $\phi\circ\alpha_1=\alpha_2\circ\phi$.
\begin{example}
Any Hom-Lie algebra is a Hom-Leibniz algebra as in the presence of skew-symmetry Hom-Jacobi identity is same as Hom-Leibniz identity.
\end{example}
\begin{example}
Given a Leibniz algebra $(L,[.,.])$ and a Leibniz algebra homomorphism $\alpha: L\to L$, one always get a Hom-Leibniz algebra $(L,[.,.]_\alpha,\alpha)$, where $[x,y]_\alpha= [\alpha(x), \alpha(y)]$.
\end{example}

\begin{example}
Let $(L, d)$ be a differential Lie algebra with the Lie bracket $[.,.]$ and $\alpha : L \to L$ is an endomorphism such that
\begin{enumerate}
\item $\alpha [x, y] = [\alpha(x), \alpha(y)],$
\item $\alpha \circ d = d \circ \alpha.$
\end{enumerate}
Then $L$ is a Hom-Leibniz algebra with the new bracket defined as
$$[x, y]_{d,\alpha} := [\alpha(x), d \alpha(y)].$$
\end{example}

\begin{example} \label{exam fd}
Let $L$ is a two-dimensional $\mathbb{C}$-vector space with basis $\lbrace e_1, e_2\rbrace$. We define a bracket as $[e_2,e_2]=e_1$ and zero else where and the endomorphism is given by the matrix \[
\alpha=
  \begin{bmatrix}
    1 & 1  \\
    0 & 1 
  \end{bmatrix}.
\]
It is easy to check that $(L, [.,.],\alpha)$ is a Hom-Leibniz algebra which is not Hom-Lie.
\end{example}
\begin{defn}
A Hom-Zinbiel algebra is a triple $(R, (), \alpha_1)$ consisting of a vector space $R$ together with two $k$-linear maps
\begin{align*}
(): R\otimes R\to R~~~\text{and}~~~ \alpha_1: R\to R,
\end{align*}
satisfying the relation
\begin{align}
((xy)\alpha_1(z))=(\alpha_1(x)(yz))+(\alpha_1(x)(zy)),
\end{align}
for all $x,y,z\in R$.
\end{defn}
For a graded vector space $R$, we can define a notion of a graded Hom-Zinbiel algebra, satisfying the following graded relation
\begin{align}
((xy)\alpha_1(z))=(\alpha_1(x)(yz))+(-1)^{|y||z|}(\alpha_1(x)(zy)).
\end{align}
\subsection{Shuffle algebra}
  Let $S_p$ be the permutation group of $p$ elements $1, \ldots, p.$ A permutation $\sigma \in S_p$ is called a $(n, m)$-shuffle if $n+m = p$ and 
$$\sigma (1) < \cdots <\sigma (n) ~~\mbox{and}~~~\sigma (n+1) < \cdots <\sigma (n+m).$$ 
In the group algebra $k [S_p],$ let $sh_{n,m}$ be the element $$sh_{n,m}: = \sum_{\sigma} \sigma,$$ where the summation is over all $(n, m)\mbox{-shuffles}.$

For any vector space $V$ we let $\sigma \in S_p$ act on $V^{\otimes p}$ by
$$\sigma(v_1\ldots v_p) = (v_{\sigma^{-1}(1)}\ldots v_{\sigma^{-1}(p)}),$$ where the generator $v_1\otimes \cdots \otimes v_p$ of $V^{\otimes p}$ is denoted by $v_1\ldots v_p.$

Note that the linear map from $k [S_p]$ to itself induced by $\sigma \mapsto \mbox{sgn}(\sigma)\sigma^{-1}$ for $\sigma \in S_p$ is an anti-homomorphism. Let us denote the image of $\alpha \in k [S_p]$  under this map by $\tilde{\alpha}.$ 
For any non-negative integers $n$ and $m$ we define 
\begin{align}\label{equivariant-linear-map}
\rho_{n,m}: = 1_1\otimes \widetilde{sh}_{n-1,m} = L^{\otimes n+m} \rightarrow L^{\otimes n+m},
\end{align}

is given by
\begin{align}\label{explicit-above-map}
\rho_{n,m}(x_1,\ldots, x_{n+m})= \sum_\sigma \mbox{sgn}(\sigma)(x_1,x_{\sigma (2)},\ldots, x_{\sigma (n+m)}),
\end{align}
where the above sum is over all $(n-1, m)$-shuffles $\sigma$. 

For the generators $b=b_1\ldots b_m,~ c=c_1\ldots c_r$. Suppose $\tau_{m,r}$ is a permutation such that
\begin{align}\label{tau}
\tau_{r,m}(cb)=bc.
\end{align}
\begin{rmk}  
 For non-negative integers $n,~~m,~~r,$ under the anti-homomorphism $\alpha\to \tilde{\alpha}$ we have the following equality 
\begin{align}\label{relation-required-for-proof}
(\rho_{n,m}\otimes  1_r) \circ \rho_{n+m, r} = ( 1_n\otimes \rho_{m,r} + (-1)^{rm}\circ \tau_{r,m} \circ \rho_{r,m})\circ \rho_{n,m+r}
\end{align}
(See Proposition (1.8) of \cite{L3})
\end{rmk}
\section{Some results on Hom-Leibniz and Hom-Zinbiel algebras}
In this section, we prove some results regarding Hom-Leibniz and Hom-Zinbiel algebras along the same line as Leibniz and Zinbiel algebras.
\begin{prop}
Suppose $(L,[.,.],\alpha)$ is a Hom-Leibniz algebra and $(R, (), \alpha_1)$ is a Hom-Zinbiel algebra. Then the tensor product $L\otimes R$ together with the bracket and $k$-linear map
\begin{align*}
&[x\otimes r, y\otimes s]=[x,y]\otimes (rs)-[y,x]\otimes (sr),\\
&\alpha\otimes \alpha_1: L\otimes R \to L\otimes R,\,\,\, (\alpha\otimes \alpha_1)(x\otimes r)=\alpha(x)\otimes \alpha_1(r),
\end{align*}
is a Hom-Lie algebra.
\end{prop}
\begin{proof}
Skew-symmetry part is trivial. Hom-Jacobi identity is same as the following relation
\begin{align*}
\sum_{\sigma \in S_3} sgn(\sigma)\sigma[\alpha(x)\otimes \alpha_1(r),[y\otimes s, z\otimes t]]=0,
\end{align*}
here permutation group $S_3$ is acting same way on both sets $x,y,z$ and $r,s,t$. We have,
\begin{align*}
\sum_{\sigma\in S_3}sgn(\sigma)\sigma &\big([\alpha(x),[y,z]]\otimes (\alpha_1(r)(st))-[[y,z],\alpha(x)]\otimes \big((st)\alpha_1(r))\\
&-[\alpha(x),[z,y]]\otimes (\alpha_1(r)(st))+[[z,y],\alpha(x)]\otimes ((ts)\alpha_1(r))\big)=0.
\end{align*}
Using Hom-Leibniz relation on $L$ this is same as
\begin{align*}
\sum_{\sigma\in S_3}sgn(\sigma)\sigma &\big(([[x,y],\alpha(z)]-[[x,z],\alpha(y)])\otimes (\alpha_1(r)(st))-[[y,z],\alpha(x)]\otimes \big((st)\alpha_1(r))\\
&-([[x,z],\alpha(y)]-[[x,y],\alpha(z)])\otimes (\alpha_1(r)(st))+[[z,y],\alpha(x)]\otimes ((ts)\alpha_1(r))\big)=0.
 \end{align*}
Note that $S_3$ acting symmetrically on the sum. Thus, it is enough to prove that the coefficient of $([x, y], \alpha(z)]]$ is $0$.
\begin{align*}
&(\alpha_1(r)(st))+(\alpha_1(r)(ts))-((rs)\alpha_1(t))+(\alpha_1(r)(st))+(\alpha_1(r)(ts))-((rs)\alpha_1(t))\\
&2((\alpha_1(r)(st))+(\alpha_1(r)(ts))-((rs)\alpha_1(t)))=0
\end{align*}
using the Hom-Zinbiel relation.
\end{proof}
\begin{prop}
Let $(R,\alpha_1)$ be a Hom-Zinbiel algebra. Then, $(R, \alpha_1, 
\ast )$ is a Hom-associative and commutative algebra, where 
$$\ast:R\otimes R\to R,\,\,\,x\ast y=xy+yx.$$ 
\end{prop}
\begin{proof}
Commutativity is trivial from the definition. We only need to check the Hom-associativity.
\begin{align*}
\alpha_1(x)\ast (y\ast z)&=\alpha_1(x)\ast \big((yz)+ (zy)\big)\\
                                    &=\alpha_1(x)\big((yz)+ (zy)\big)+ \big((yz)+ (zy)\big)\alpha_1(x)\\
                                    &=\alpha_1(x)\big((yz)+ (zy)\big)+ (yz)\alpha_1(x)+ (zy)\alpha_1(x)\\
                                    &=\alpha_1(x)\big((yz)+ (zy)\big)+\alpha_1(y)\big((zx)+(xz)\big)+ \alpha_1(z)\big((yx)+ (xy)\big).
\end{align*}
\begin{align*}
(x\ast y)\ast \alpha_1(z)&=\big((xy)+ (yx)\big)\ast \alpha_1(z)\\
                                    &=\big((xy)+ (yx)\big)\alpha_1(z)+ \alpha_1(z)\big((xy)+ (yx)\big)\\
                                    &=(xy)\alpha_1(z)+ (yx)\alpha_1(z)+ \alpha_1(z)\big((xy)+ (yx)\big)\\
                                    &=\alpha_1(x)\big((yz)+ (zy)\big)+ \alpha_1(y)\big((xz)+ (zx)\big)+ \alpha_1(z)\big((xy)+ (yx)\big).         
\end{align*}
Thus, $(R, \alpha_1,\ast )$ is a Hom-associative, commutative algebra.
\end{proof}

\section{Cohomology of Hom-Leibniz algebras}
  In this section, we introduce a cohomology for Hom-Leibniz algebras with coefficients in Hom-associative and commutative algebras. To define cochain
complex suitable for our work, we need to consider multiplicative Hom-Leibniz algebras. From now on by Hom-Leibniz algebra we always mean a multiplicative Hom-Leibniz algebra.
  
Suppose $(L,[.,.],\alpha)$ is a Hom-Leibniz algebra. Denote $CL_n(L)=L^{\otimes n}$ and define 
$$d_n:CL_n(L)\to CL_{n-1}(L),$$
where 
\begin{align}\label{boundary-map}
d_n(x_1,\ldots,x_n)=\sum_{1\leq i<j\leq n}(-1)^{j+1}\big(\alpha(x_1),\ldots,\alpha(x_{i-1}),[x_i,x_j],\alpha(x_{i+1}),\ldots,\widehat{\alpha(x_j)},\ldots,\alpha(x_n)\big).
\end{align}
\begin{theorem}
For all $n\geq 1$, $d_{n-1}d_n=0$.
\end{theorem}
\begin{proof}
Proof follows from the Theorem (3.1) of \cite{CS}.
\end{proof}
Thus, we have a chain complex
$$CL_\sharp(L): \cdots \rightarrow L^{\otimes n}\stackrel{d_n}{\rightarrow} L^{\otimes (n-1)}\stackrel{d_{n-1}}{\rightarrow}\cdots \stackrel{d_3}{\rightarrow} L^{\otimes 2} \stackrel{d_2}{\rightarrow} L .$$
The homology groups of this chain complex are called homology groups of the Hom-Leibniz algebra $(L,[.,.],\alpha)$ and $n$th homology group is denoted by $HL^\alpha_n(L)$. Note that for $\alpha=id$, $HL^\alpha_n(L)$ is same as the Leibniz homology of $L$ defined in \cite{L1}.
Suppose $(L,[.,.],\alpha)$ is a Hom-Leibniz algebra and let $(A,\mu, \alpha_1)$ be a Hom-associative and commutative algebra over $k$.

Set $CL_{\alpha, \alpha_1}^n(L;~A) =\lbrace f: L^{\otimes n}\to A\mid \alpha_1\circ f= f\circ \alpha^{\otimes n}\rbrace $, for all $n\geq 1$. We define $$\delta^n_{\alpha} : CL^n_{\alpha, \alpha_1}(L;~A)\rightarrow CL_{\alpha, \alpha_1}^{n+1}(L;~A)$$ by $\delta^n_\alpha (c) = c\circ d_{n+1},$ $c \in CL_{\alpha, \alpha_1}^n(L;~A),$ where $d_{n+1} : L^{\otimes(n+1)} \rightarrow L^{\otimes n}$ is the boundary map (\ref{boundary-map}). Explicitly, for any $c \in CL_{\alpha, \alpha_1}^n(L;~A)$ and $(x_1, \ldots ,x_{n+1}) \in L ^{\otimes (n+1)},$
$\delta^n_\alpha (c)(x_1, \ldots ,x_{n+1})$ is given by the expression
\begin{align}\label{coboundary-map}
\sum_{1\leq i < j \leq n} (-1)^{j+1} c\big(\alpha(x_1),\cdots,\alpha(x_{i-1}),[x_i,x_j],\alpha(x_{i+1}),\ldots,\widehat{\alpha(x_j)},\ldots,\alpha(x_{n+1})\big).
\end{align}
Note that $\delta^n_\alpha (c)$ need to satisfy the condition $\alpha_1\circ \delta^n_\alpha (c)= \delta^n_\alpha (c)\circ \alpha^{\otimes n}$ to be a member of $CL_{\alpha, \alpha_1}^{n+1}(L;~A)$. This is why we need to consider multiplicative Hom-Leibniz algebras to define cochain complex in this way.

Clearly, $\delta_\alpha^2 =0$ as $d^2=0,$ and therefore, $(CL_{\alpha, \alpha_1}^\sharp (L;~A), \delta)$ is a cochain complex. Its homology groups are called Hom-Leibniz cohomology groups of $L$ with coefficients in $A$ and $n$th cohomology group is denoted by $HL_{\alpha, \alpha_1}^n(L;~A).$
\section{Cup-product in Hom-Leibniz cohomology}
Main aim of this section is to define a cup-product operation on the graded cohomology of the Hom-Leibniz algebra and show that the cup-product satisfies the graded Hom-Zinbiel relation.

\begin{prop} 
Given a Hom-associative and commutative algebra $(A,\mu, \alpha_1)$ there is a induced map
\begin{align*}
\alpha_1: HL^n_{\alpha, \alpha_1}(L;~A)\to HL^n_{\alpha, \alpha_1}(L;~A),
\end{align*} 
for all $n\geq 1$.
\end{prop}
\begin{proof}
Using the twisting map $\alpha_1$ of $A$ we define a map 
\begin{align*}
\underline{\alpha_1}: CL_{\alpha, \alpha_1}^n(L;&~A)\to CL_{\alpha, \alpha_1}^n(L; ~A),\\
&f\mapsto \alpha_1\circ f=f\circ \alpha^{\otimes n}.
\end{align*}
Suppose $f\in  CL_{\alpha, \alpha_1}^n(L;~A)$ is a cocycle. This implies $\delta^n_\alpha f=0$. We need to show $\delta^n_\alpha (\alpha_1\circ f)=0$.  For all $(x_1\ldots,x_{n+1})$, we have,
\begin{align*}
&\delta^n_\alpha (\alpha_1\circ f)(x_1\ldots,x_{n+1})\\
&=\sum_{1\leq i<j\leq n+1}(\alpha_1\circ f)(\alpha(x_1),\ldots,\alpha(x_{i-1}),[x_i,x_j],\alpha(x_{i+1}),\ldots,\widehat{\alpha(x_j)},\ldots,\alpha(x_{n+1}))\\
                                                                                     &=\alpha_1\Big(\sum_{1\leq i<j\leq n+1} f(\alpha(x_1),\ldots,\alpha(x_{i-1}),[x_i,x_j],\alpha(x_{i+1}),\ldots,\widehat{\alpha(x_j)},\ldots,\alpha(x_{n+1}))\Big)\\
                                                                                     &=\alpha_1(\delta^n_\alpha f)=0.
\end{align*}
Now suppose $f=\delta^{n-1}_\alpha  g$ for some $g\in CL_{\alpha, \alpha_1}^{n-1}(L;~A)$. Using the linearity of $\alpha_1$ we have $\alpha_1\circ f=\alpha_1\circ (\delta^{n-1}_\alpha  g)=\delta^{n-1}_\alpha  (\alpha_1\circ g)$.

Thus, under the image of $\underline{\alpha_1}$ cocycles goes to cocycles and coboundaries goes to coboundaries. So $\underline{\alpha_1}$ induces a map on the cohomology level
\begin{align*}
\alpha_1: HL^n_{\alpha, \alpha_1}(L;~A)\to HL^n_{\alpha, \alpha_1}(L;~A).
\end{align*}
Note that we have denoted the induced map on the cohomology level by the same notation as for the twisting map $\alpha_1$ of $A$.
\end{proof}

\begin{lemma}\label{lemma main}
Given a Hom-Leibniz algebra $(L,[.,.],\alpha)$ and a Hom-associative, commutative algebra $(A,\mu, \alpha_1)$, there is a bilinear operation on the graded space $CL_{\alpha, \alpha_1}^\ast (L;~A)$, 
\begin{align*}
\cup : CL_{\alpha, \alpha_1}^n(L;~A)\times CL_{\alpha, \alpha_1}^m(L;~A)\rightarrow CL_{\alpha, \alpha_1}^{n+m}(L;~A),
\end{align*}
defined as
\begin{align}
f\cup g &:=\mu\circ (\alpha_1^{m-1}f\otimes \alpha^{n-1}_1g)\circ \rho_{n,m}\\
             &:=\mu\circ (f\circ (\alpha^{m-1})^{\otimes n}\otimes g\circ (\alpha^{n-1})^{\otimes m})\circ \rho_{n,m},
\end{align}
for $f\in CL_\alpha^n(L;~A)$ and $g\in CL_\alpha^m(L;~A)$
and the operation $\cup$ satisfies
\begin{align}\label{cup well define}
\delta_\alpha(f\cup g)=\delta_\alpha(f)\cup g+~(-1)^{|f|} f\cup \delta_\alpha(g).
\end{align}
\end{lemma}
\begin{proof}
To prove
\begin{align}
\delta_\alpha(f\cup g)=\delta_\alpha(f)\cup g+~(-1)^{|f|} f\cup \delta_\alpha(g).
\end{align}
It is enough to show the following
\begin{align}\label{well define cond}
&\big((\alpha^{m-1})^{\otimes n}\otimes (\alpha^{n-1})^{\otimes m}\big)\circ \rho_{n,m}\circ d_{n+m+1}\nonumber \\
&=~ \big((d_{n+1}\circ (\alpha^{m-1})^{\otimes{n+1}})\otimes (\alpha^n)^{\otimes m}\big)\circ \rho_{n+1,m}\\
&+(-1)^n \big((\alpha^m)^{\otimes n}\otimes (d_{m+1}\circ (\alpha^{n-1})^{\otimes m+1})\big)\circ \rho_{n,m+1}\nonumber.
\end{align}
We define
\begin{align*}
\delta^{i,j}_\alpha: L^{\otimes p}\to L^{\otimes p-1},~~~\text{for}~~~1\leq i<j\leq p
\end{align*}
as $\delta^{i,j}_\alpha(x_1,\ldots,x_p):=\big(\alpha(x_1),\ldots,\alpha(x_{i-1}),[x_i,x_j],\alpha(x_{i+1}),\ldots,\widehat{\alpha(x_j)},\ldots,\alpha(x_n)\big)$.
Note that, $d_p=\sum_{1\leq i<j\leq p}(-1)^j\delta^{i,j}_\alpha$.

Consider
\begin{align}\label{case1}
\big((\delta^{k,l}_\alpha\circ (\alpha^{m-1})^{\otimes{n+1}})\otimes (\alpha^n)^{\otimes m}\big)(1_1\otimes \sigma^{-1}),
\end{align}where $1\leq k<l \leq {n+1}$ and $\sigma$ is a $(n,m)$-shuffle acting on $\lbrace 2,\ldots,n+m+1\rbrace$. This operator is a part of $\big((d_{n+1}\circ (\alpha^{m-1})^{\otimes{n+1}})\otimes (\alpha^n)^{\otimes m}\big)\circ \rho_{n+1,m}$. 

As $\sigma$ is a $(n,m)$-shuffle, $\sigma(k)<\sigma(l)$. The equation (\ref{case1}) is same as the following equation
\begin{align}\label{case 1 equ}
\big((\alpha^{m-1})^{\otimes n}\otimes (\alpha^{n-1})^{\otimes m}\big)(1_1\otimes \omega^{-1})\delta^{\sigma(k),\sigma(l)}_\alpha
\end{align}
for some permutation $\omega$. As $l\leq n+1$ and $\sigma$ is a $(n, m)$-shuffle, $\omega$ is a $(n-1, m)$-shuffle.

Thus, equation (\ref{case 1 equ}) is a part of $\big((\alpha^{m-1})^{\otimes n}\otimes (\alpha^{n-1})^{\otimes m}\big)\circ \rho_{n,m}\circ d_{n+m+1}$.

Now consider,
\begin{align}\label{case 2}
 \big((\alpha^m)^{\otimes n}\otimes (\delta^{k,l}_\alpha\circ (\alpha^{n-1})^{\otimes m+1})\big)(1_1\otimes \sigma^{-1}),
\end{align}
where $n+1\leq k<l\leq n+m+1$ and $\sigma$ is a $(n-1,m+1)$-shuffle acting on $\lbrace 2,\ldots,n+m+1\rbrace$. This operator is a part of $\big((\alpha^m)^{\otimes n}\otimes (d_{m+1}\circ (\alpha^{n-1})^{\otimes m+1})\big)\circ \rho_{n,m+1}$. As $\sigma$ is a $(n-1,m+1)$-shuffle, $\sigma(k)<\sigma(l).$ The equation (\ref{case  2}) is same as the following equation
\begin{align}\label{case 2 equ}
\big((\alpha^{m-1})^{\otimes n}\otimes (\alpha^{n-1})^{\otimes m}\big)(1_1\otimes \omega^{-1})\delta^{\sigma(k),\sigma(l)}_\alpha
\end{align}
for some permutation $\omega$. As $k\geq n+1$ and $\sigma$ is a $(n, m)$-shuffle, $\omega$ is a $(n-1, m)$-shuffle.

Thus, equation (\ref{case 2 equ}) is a part of $\big((\alpha^{m-1})^{\otimes n}\otimes (\alpha^{n-1})^{\otimes m}\big)\circ \rho_{n,m}\circ d_{n+m+1}$.

Using the same argument from \cite{L3}, it is easy to check that both sides of (\ref{well define cond}) contains the same number of elements. We just showed that terms in the right hand side belongs to terms in the left hand side. This prove the validity of the equation (\ref{well define cond}). Therefore, the cup product is well-defined. 
\end{proof}

\begin{theorem}\label{func}
Let $(L_1, [.,.], \alpha)$ and $(L_2, [.,.], \bar{\alpha})$ be two Hom-Leibniz algebras and $(A, \mu, \alpha_1)$ be any Hom-associative, commutative algebra. The cup-product is functorial, that is, if $\phi : L_1 \to L_2$ is a homomorphism then the induced map 
$$\phi^\ast : HL^n_{\bar{\alpha}, \alpha_1}(L_2;~A) \to HL^n_{\alpha, \alpha_1}(L_1;~A),~~~\phi^\ast (f) \mapsto f \circ \phi^{\otimes n},$$
commutes with $\cup$, that is, 
$$\phi^\ast (f \cup g) = \phi^\ast (f) \cup \phi^\ast(g).$$
\end{theorem}
\begin{proof}
Let $f \in CL^n_{\bar{\alpha}, \alpha_1}(L_2;~A)$ and $g \in CL^m_{\bar{\alpha}, \alpha_1}(L_2;~A)$. To prove the theorem, it is sufficient to show
$$\phi^\ast (f \cup g) = \phi^\ast (f) \cup \phi^\ast (g)$$
on cochains.
\begin{align*}
\phi^\ast (f) \cup \phi^\ast (g) &= \mu\circ (\alpha_1^{m-1} \phi^\ast (f)\otimes \alpha^{n-1}_1 \phi^\ast (g))\circ \rho_{n,m}  \\
                                                  &= \mu\circ (\alpha_1^{m-1} f \circ \phi^{\otimes n}\otimes \alpha^{n-1}_1 g \circ \phi^{\otimes m})\circ \rho_{n,m}\\
                                                  &= \mu\circ (\alpha_1^{m-1} f \otimes \alpha^{n-1}_1 g)\circ \rho_{n,m} \circ \phi^{\otimes n+m}\\
                                                  &= \phi^\ast (f \cup g).
\end{align*}
\end{proof}
\begin{cor}
If $\phi : L_1 \to L_2$ is an isomorphism then the induced map on the cohomology level 
$$\phi^\ast : HL^n_{\bar{\alpha}, \alpha_1}(L_2;~A) \to HL^n_{\alpha, \alpha_1}(L_1;~A),$$
is also an isomorphism.
\end{cor}

\begin{remark}
If $f_2 \cup f_2 =0$, that is, $f$ is a square zero element for some $f_2 \in HL^n_{\bar{\alpha}, \alpha_1}(L_2;~A)$ and $f_1 \cup f_1 \neq 0$ for any $f_1 \in HL^n_{\alpha, \alpha_1}(L_1;~A)$. Then we can conclude that $L_1$ and $L_2$ can not be ismorphic as Hom-Leibniz algebras. In this way the cup-product can be used to distinguish Hom-Leibniz algebras.
\end{remark}

Now we show the cup-product satisfes the graded Hom-Zinbiel relation. 
\begin{theorem}
The operation $\cup$ induces a well defined bilinear product on the graded cohomology, called the cup-product, 
\begin{align*}
\cup : HL^n_{\alpha, \alpha_1}(L;~A)\times HL^m_{\alpha, \alpha_1}(L;~A)\rightarrow HL^{n+m}_{\alpha, \alpha_1}(L;~A),
\end{align*}
 and the cup product $\cup$ together with the linear map
 \begin{align*}
\alpha_1: HL^n_{\alpha, \alpha_1}(L;~A)\to HL^n_{\alpha, \alpha_1}(L;~A),
\end{align*}
satisfies  the  following graded Hom-Zinbiel formula
\begin{align*}
\alpha_1(a\cup b)\cup c= a\cup\alpha_1(b\cup c)+(-1)^{|b||c|}a\cup\alpha_1(c\cup b),
\end{align*}
where $a\in HL^n_{\alpha, \alpha_1}(L;~A),~b\in HL^m_{\alpha, \alpha_1}(L;A),~c\in HL^r_{\alpha, \alpha_1}(L;~A)$.
\end{theorem}
\begin{proof}
Suppose $a,b,c$ is represented by cocycles $f,g,h$ respectively. From the formula (\ref{cup well define}) of Lemma \ref{lemma main} it is clear that if $f, g$ are cocycles then $f\cup g$ is also a cocycle and if either $f$ or $g$ is a coboundary then $f\cup g$ is a coboundary. Thus, the cup-product defined on the cochain level induces a cup-product on the cohomology level.

Since $A$ is commutative, we have,
\begin{align*}
\mu\circ (h\circ (\alpha^{m-1})^{\otimes r}\otimes g\circ (\alpha^{r-1})^{\otimes m})\circ \rho_{r,m}=\mu\circ (g\circ (\alpha^{r-1})^{\otimes m})\otimes h\circ (\alpha^{m-1})^{\otimes r}) \circ \tau_{r,m}\circ \rho_{r,m}.
\end{align*}
Where $\tau_{r,m}$ is a permutaion defined in (\ref{tau}).
Now, pre-composing both sides of (\ref{relation-required-for-proof}) by $\mu\circ((f(\alpha^{m+r-1})^{\otimes n}\otimes g(\alpha^{n+r-1})^{\otimes m})\otimes h(\alpha^{m+n-1})^{\otimes r})$
\begin{align*}
&\mu\circ((f(\alpha^{m+r-1})^{\otimes n}\otimes g(\alpha^{n+r-1})^{\otimes m})\otimes h(\alpha^{m+n-1})^{\otimes r})\circ (\rho_{n,m}\otimes  1_r) \circ \rho_{n+m,r}\\
& =\mu\circ((f(\alpha^{m+r-1})^{\otimes n}\otimes g(\alpha^{n+r-1})^{\otimes m})\otimes h(\alpha^{m+n-1})^{\otimes r})\circ (1_n\otimes \rho_{m,r})\circ \rho_{n,m+r}\\
& + (-1)^{rm} \mu\circ((f(\alpha^{m+r-1})^{\otimes n}\otimes g(\alpha^{n+r-1})^{\otimes m})\otimes h(\alpha^{m+n-1})^{\otimes r})\circ (1_n\otimes \tau_{r,m}) \circ \rho_{r,m}\circ \rho_{n,m+r}.
\end{align*}
This is same as 
\begin{align*}
&\mu\circ((f(\alpha^{m-1})^{\otimes n}\otimes g(\alpha^{n-1})^{\otimes m})(\alpha^r)^{\otimes m+n}\otimes h(\alpha^{m+n-1})^{\otimes r})\circ (\rho_{n,m}\otimes  1_r) \circ \rho_{n+m,r}\\
& =\mu\circ(f(\alpha^{m+r-1})^{\otimes n}\otimes (g(\alpha^{r-1})^{\otimes m}\otimes h(\alpha^{m-1})^{\otimes r})(\alpha^{n})^{\otimes m+r}\circ (1_n\otimes \rho_{m,r})\circ \rho_{n,m+r}\\
& + (-1)^{rm} \mu\circ(f(\alpha^{m+r-1})^{\otimes n}\otimes (g(\alpha^{r-1})^{\otimes m}\otimes h(\alpha^{m-1})^{\otimes r})(\alpha^{n})^{\otimes m+r}\circ (1_n\otimes \tau_{r,m}) \circ \rho_{r,m}\circ \rho_{n,m+r}.
\end{align*}
This implies
\begin{align*}
&\mu\circ(\alpha_1(f(\alpha^{m-1})^{\otimes n}\otimes g(\alpha^{n-1})^{\otimes m})(\alpha^{r-1})^{\otimes m+n}\otimes h(\alpha^{m+n-1})^{\otimes r})\circ (\rho_{n,m}\otimes  1_r) \circ \rho_{n+m,r}\\
& =\mu\circ(f(\alpha^{m+r-1})^{\otimes n}\otimes \alpha_1(g(\alpha^{r-1})^{\otimes m}\otimes h(\alpha^{m-1})^{\otimes r})(\alpha^{n-1})^{\otimes m+r}\circ (1_n\otimes \rho_{m,r})\circ \rho_{n,m+r}\\
& + (-1)^{rm} \mu\circ(f(\alpha^{m+r-1})^{\otimes n}\otimes \alpha_1(g(\alpha^{r-1})^{\otimes m}\otimes h(\alpha^{m-1})^{\otimes r})(\alpha^{n-1})^{\otimes m+r}\circ (1_n\otimes \tau_{r,m}) \circ \rho_{r,m}\circ \rho_{n,m+r}.
\end{align*}
Thus,
\begin{align*}
\alpha_1(f\cup g)\cup h= f\cup\alpha_1(g\cup h)+(-1)^{|g||h|}f\cup\alpha_1(h\cup g).
\end{align*}
\end{proof}

\section{A simple computation}
We illustrate our theory by giving a simple computation which is based on the Example \ref{exam fd} with coefficients in a Hom-associative and commutative algebra given in the Example \ref{hom ass exam}. Here we first recall Examples \ref{exam fd} and \ref{hom ass exam}.

Let $L$ be a two dimensional $\mathbb{C}$-vector space with basis $\lbrace e_1, e_2 \rbrace$. Then $L$ is a Hom-Leibniz algebra with respect to the following bracket and endomorphisms:
$$[e_2, e_2] = e_1~~~\text{and zero elsewhere.}$$
The endomorphism $\alpha$ is given by
\begin{align*}
& \alpha(e_1)= e_1,\\
& \alpha (e_2) =e_1 + e_2.
\end{align*}
Let $A$ be a two dimensional vector space with basis $\lbrace a_1, a_2 \rbrace$. The multiplication $\mu$ of the Hom-associative and commutative algebra $A$ is defined as:
$$ \mu (a_i, a_j) = \begin{cases} a_1 & \mbox{ if } (i,j) = (1,1)\\
a_2 & \mbox{ if } (i,j) \neq (1,1).
\end{cases} $$
The endomorphism $\alpha_1 :A \rightarrow A$ is defined by $\alpha_1 (a_1) = a_1 - a_2$ and $\alpha_1 (a_2) = 0$.

We have the following chain complex:
$$CL_\sharp(L): \cdots \rightarrow L^{\otimes n}\stackrel{d_n}{\rightarrow} L^{\otimes (n-1)}\stackrel{d_{n-1}}{\rightarrow}\cdots \stackrel{d_3}{\rightarrow} L^{\otimes 2} \stackrel{d_2}{\rightarrow} L .$$

We write boundary maps $d_n$ for $n\leq 3$,
\begin{align*}
& d_n =0~~~\text{for}~~~n \leq 1,~~~d_2(x_1, x_2)= [x_1, x_2],\\
& d_3 (x_1, x_2, x_3) = - ([x_1, x_2], \alpha(x_3)) + ([x_1, x_3], \alpha(x_2)) + (\alpha(x_1), [x_2, x_3]).
\end{align*} 
To understand the boundary maps $d_2, d_3$, it is enough to write down boundary maps on the basis elements. Here we write only non-zero cases.
\begin{align*}
& d_2(e_2, e_2) = [e_2, e_2] = e_1,\\
& d_3 (e_1, e_2, e_2) = (\alpha(e_1), [e_2, e_2])= (e_1, e_1),\\
& d_3 (e_2, e_2, e_1) = -([e_2, e_2], \alpha(e_1))= - (e_1, e_1),\\
& d_3 (e_2, e_2, e_2) = (\alpha(e_2), [e_2, e_2])= (e_1 + e_2, e_1).
\end{align*}
We define a map $f : L \to A$ as $f(e_2) = a_2-a_1,~~~f(e_1)=0.$

We need to check that $\alpha_1 \circ f = f \circ \alpha$. 
\begin{align*}
& (\alpha_1 \circ f)(e_1) = \alpha_1(f(e_1))= \alpha_1(0)=0,\\
& (f\circ \alpha)(e_1) = f (\alpha(e_1)) = f(e_1)= 0,\\
& (\alpha_1 \circ f)(e_2) = \alpha_1(f(e_2))= \alpha_1(a_2 - a_1)= a_2-a_1,\\
& (f\circ \alpha)(e_2) = f (\alpha(e_2)) = f(e_1+ e_2)=f(e_2)= a_2 -a_1.
\end{align*}
Thus, $f \in CL^1_{\alpha, \alpha_1}(L,~A)$.
Now, 
$$f\cup f= \mu \circ (f\otimes f) \circ \rho_{1,1} \in CL^2_{\alpha, \alpha_1}(L,~A).$$
Note that
\begin{align*}
&\rho_{1, 1}(e_1, e_1) =~ \rho_{1,1}(e_2, e_2)=0,\\
&\rho_{1, 1}(e_1, e_2) = (e_1, e_2),\\
&\rho_{1,1} (e_2, e_1) = (e_2, e_1).
\end{align*}
One can check easily from the definition of the cup-product that
\begin{align*}
& (f\cup f) (e_1, e_1) =0,\\
& (f\cup f) (e_1, e_2) =0,\\
& (f\cup f) (e_2, e_1) = 0,\\
& (f\cup f) (e_2, e_2) =0.
\end{align*} 
Thus, $$f \cup f = 0.$$
\begin{remark}
It may be remarked that although cohomology operations are well understood in topology and these operations are useful in distingushing spaces, however these operations have not been explored well in the context of cohomology of algebras. Having a cup-product is the first step towards it. In a future project we intend to address this problem in general, and hence,  cup-product in cohomology of Leibniz algebras or Hom Leibniz algebras are being studied at present. Further, any new algebraic structure on cohomology of a type of algebras helps in the classification problem of the algebras of a given type. This is one of the motivation behind introducing cup-product.
\end{remark}

{\bf Acknowledgements:} 
The author express his gratitute to the esteemed referee for her/his useful comments and suggestions on the earlier version of the manuscript that have improved the exposition. The author would like to thank Prof. Goutam Mukherjee of Indian Statistical Institute, Kolkata, for his guidance and reading a draft version of this article. The author is also thankful to Apurba Das and Surojit Ghosh for their help and useful comments on this article.


\begin{thebibliography}{BFGM03}
\bibitem{CS}
Y. S. Cheng and Y. C. Su, (Co)homology and universal central extensions of Hom-Leibniz algebras, \emph{Acta Math. Sin. (Engl. Ser.)} \textbf{27} (2011), no. 5, 813-830.
\bibitem{AD}
A. Das, Gerstenhaber algebra structure on the cohomology of a hom-associative algebra, \emph{arXiv:1805.01207}
\bibitem{G1}
M. Gerstenhaber, The cohomology structure of an associative ring, \emph{Ann. of Math.} \textbf{78} (1963) 267-288.
\bibitem{HLS}
J.T.. Hartwig, D. Larsson and S. D. Silvestrov, Deformations of Lie algebras using $\sigma$-derivations, \emph{J. Algebra} \textbf{295} (2006), 314-361.
\bibitem{L1}
J.-L. Loday, Une version non commutative des alg$\grave{e}$bres de Lie: les alg$\grave{e}$bres de Leibniz, \emph{Ens. Math.} \textbf{39} (3-4) (1993) 269-293.
\bibitem{L3}
J.-L. Loday, Cup-product for Leibniz cohomology and dual Leibniz algebras,  \emph{Math. Scand.} \textbf{77} (2) (1995) 189-196.

\bibitem{MS08}
A. Makhlouf and S.D. Silvestrov. Hom-algebra structures, \emph{J. Gen. Lie
Theory Appl.}, \textbf{2} (2008), 51-64.

\bibitem{MZ}
A. Makhlouf, P. Zusmanovich, Hom-Lie structures on Kac–Moody algebras, \emph{Journal of Algebra}, Volume 515, 2018, 278-297.

\bibitem{MS} 
G. Mukherjee, R. Saha, Cup-product for equivariant Leibniz cohomology and Zinbiel algebras, \emph{Algebra Colloquium}, \textbf{26} : 2 (2019), 271-284.

\bibitem{DY}
D. Yau, Enveloping algebra of Hom-Lie algebras, \em{J. Gen. Lie Theory Appl.} \textbf{2} (2), 95-108 (2008).

\end{thebibliography}
\end{document}